\documentclass[12pt]{amsart}
\usepackage{amssymb}
\usepackage{mathtools}
\usepackage[english]{babel}
\usepackage{epsfig}
\usepackage{mathptmx}
\usepackage{amsmath,amsfonts,amssymb}
\usepackage{mathtools}
\usepackage{mathrsfs}
\usepackage[all]{xy}
\usepackage{graphicx}
\usepackage{latexsym}
\usepackage{verbatim}

\numberwithin{equation}{section}

\def\Re{{\sf Re}\,}
\def\Im{{\sf Im}\,}

\newcommand{\D}{\mathbb D}

\newcommand{\R}{\mathbb R}
\newcommand{\Ha}{\mathbb H}

\newcommand{\C}{\mathbb C}

\newcommand{\Aut}{{\sf Aut}(\mathbb D)}
\newcommand{\oD}{\overline{\mathbb D}}

\newcommand{\N}{\mathbb N}

\def\Re{{\sf Re}\,}
\def\Im{{\sf Im}\,}

\newcommand{\strip}{\mathbb{S}}

\def\Aut{{\sf Aut}}

\def\Re{{\sf Re}\,}
\def\Im{{\sf Im}\,}

\def\Re{{\sf Re}\,}
\def\Im{{\sf Im}\,}

\def\1#1{\overline{#1}}
\def\2#1{\widetilde{#1}}
\def\3#1{\widehat{#1}}
\def\4#1{\mathbb{#1}}
\def\5#1{\frak{#1}}
\def\6#1{{\mathcal{#1}}}

\def\Re{{\sf Re}\,}
\def\Im{{\sf Im}\,}

\newcommand{\mcite}[1]{\csname b@#1\endcsname}

\theoremstyle{theorem}

\def\Aut{{\sf Aut}}

\def\Re{{\sf Re}\,}
\def\Im{{\sf Im}\,}

\newtheorem{theorem}{Theorem}[section]
\newtheorem{lemma}[theorem]{Lemma}
\newtheorem{proposition}[theorem]{Proposition}
\newtheorem{corollary}[theorem]{Corollary}

\theoremstyle{definition}
\newtheorem{definition}[theorem]{Definition}

\theoremstyle{remark}
\newtheorem{remark}[theorem]{Remark}

\numberwithin{equation}{section}

\title[Speeds of convergence]{Speeds of convergence of orbits of non-elliptic semigroups of holomorphic self-maps of the unit disc}

\author[F. Bracci]{Filippo Bracci}
\address{F. Bracci: Dipartimento di Matematica, Universit\`a di Roma ``Tor Vergata", Via della Ricerca
Scientifica 1, 00133, Roma, Italia.} \email{fbracci@mat.uniroma2.it}

\dedicatory{To my friend Yuri Kozitsky in occasion of his 70th birthday}


\thanks{Partially supported by PRIN {\sl Real and Complex
Manifolds: Topology, Geometry and holomorphic dynamics} n.2017JZ2SW5 and  by the MIUR Excellence Department Project awarded to the
Department of Mathematics, University of Rome Tor Vergata, CUP E83C18000100006}

\long\def\REM#1{\relax}

\begin{document}
\maketitle
\tableofcontents

\selectlanguage{english}
\begin{abstract}
We introduce  three quantities related to orbits of non-elliptic continuous semigroups of holomorphic self-maps of the unit disc, the total speed, the orthogonal speed and the tangential speed and show how they are related and what can be inferred from those. 
\end{abstract}

\section{Introduction}

Continuous semigroups of holomorphic self-maps of the unit disc $\D$, or for short, semigroups in $\D$, have been studied since the beginning of the previous century and are still a subject of interest, from the dynamical point of view, the analytic point of view and the geometric point of view, and also, for different applications.

In this paper, we consider {\sl non-elliptic} semigroups in $\D$. For such a non-elliptic semigroup $(\phi_t)$ it is well known that there exists a unique point $\tau\in\partial \D$, the Denjoy-Wolff point of $(\phi_t)$, such that the orbits of $(\phi_t)$ converges to $\tau$ uniformly on compacta. 

The main focus of this paper is to attach to any non-elliptic semigroup in $\D$, three quantities, that we call {\sl speeds}, which have interesting properties according to the type and the dynamics of the semigroup. 

The first quantity, the {\sl total speed} $v(t)$, is nothing but the hyperbolic distance $\omega(0,\phi_t(0))$ of $\phi_t(0)$ from the origin, for $t\geq 0$. This quantity is pretty much related to  the {\sl divergence rate} as defined in \cite{AroBra16}, and, indeed, the quotient $v(t)/t$ always converges as $t\to\infty$ to the so-called {\sl spectral value} of the semigroup. In particular, for parabolic semigroups, $v(t)/t\to 0$ as $t\to\infty$. We show with an example of a parabolic semigroup of zero hyperbolic step, whose orbits converge non-tangentially to the Denjoy-Wolff point, that for parabolic semigroups there is no better estimate, namely, $v(t)$ converges to $\infty$ at a speed which is always less than $t$ but can be as close to $t$ as wanted.

The total speed is always bounded from below by $-1/4\log t$, in the sense that $\liminf[v(t)-\frac{1}{4}\log t]>-\infty$. However, for hyperbolic semigroups, $1/4\log t$ can be replaced by $(\lambda/2) t$ (where $\lambda>0$ is the spectral value) and, for parabolic semigroups of positive hyperbolic step, by $\log t$. 

The total speed can be decomposed, up to a universal additive constant, as the sum of two other quantities, the {\sl orthogonal speed} $v^o(t)$ and the {\sl tangential speed} $v^T(t)$. This is a general fact of hyperbolic geometry which we prove in Section 3: given a curve $\gamma:[0,+\infty)\to \D$  starting from $0$, converging to point $\sigma\in\partial \D$, the {\sl orthogonal projection} of $\gamma(t)$ over $(-1,1)\sigma$ is the (unique) point  $\pi(\gamma(t))\in (-1,1)\sigma$ such that 
\[
\omega(\pi(\gamma(t)), \gamma(t))=\inf\{\omega(r\sigma,\gamma(t)): r\in (-1,1) \}. 
\]
Then, for all $t\geq 0$,
\[
 \omega(\pi(\gamma(t)),\gamma(t))+\omega(0,\pi(\gamma(t)))-\frac{1}{2}\log 2\leq \omega(0,\gamma(t))\leq  \omega(\pi(\gamma(t)),\gamma(t))+\omega(0,\pi(\gamma(t))).
\]
Since $(-1,1)\sigma$ is a geodesic for the hyperbolic distance, the previous formula can be considered a sort of Pytaghoras' theorem.

In case of a non-elliptic semigroup $(\phi_t)$, we define the {\sl orthogonal speed} $v^o(t):=\omega(0,\pi(\phi_t(0)))$, where $\pi$ is the orthogonal projection on $(-1,1)\tau$, where $\tau$ is the Denjoy-Wolff point of $(\phi_t)$. We also define the {\sl tangential speed} $v^T(t):=\omega(\phi_t(0), \pi(\phi_t(0)))$. By the previous formula,
\[
v(t)\sim v^o(t)+v^T(t),
\]
where, here, $\sim$ means that they have the same asymptotic behavior. 

The tangential speed is related to the slope of convergence of orbits. In particular, $v^T(t)\leq C$ for some $C>0$ and for all $t\geq 0$ if and only if the orbit $[0,\infty)\ni t\mapsto \phi_t(0)$ converges non-tangentially to the Denjoy-Wolff point.

For semigroups, another interesting relation holds, namely, for all $t\geq 0$,
\[
v^T(t)\leq v^o(t)+4\log 2.
\]
The previous inequalities imply also that there exist universal constants $C_1, C_2\in \R$ such that
\[
v^o(t)+C_1\leq v(t)\leq 2v^o(t)+C_2
\]
for all $t\geq 0$.

The previous definitions of speeds have  Euclidean counterparts and some previous results can be translated in terms of speeds using such a dictionary. It turns out that, for instance, a  recent result of D. Betsakos \cite{Bet} can be rephrased in terms of speeds, namely, for all non-elliptic semigroups, $v^o(t)\geq \frac{1}{4}\log t+C$ for all $t\geq 0$ and a constant $C\in \R$ (while, for parabolic semigroups of positive hyperbolic step, $1/4\log t$ can be replaced by $1/2\log t$). 

Besides settling the notions of speeds and proving the aforementioned results, in this paper we provide a direct computation of total, orthogonal and tangential speeds in some cases (essentially when the image of the Koenigs function is a vertical angular sector). 

The paper ends with a section of open questions which naturally arise from the developed theory.

\section{Hyperbolic geometry in simply connected domain}

Let $\D:=\{\zeta\in\C: |\zeta|<1\}$. We denote by $\kappa_\D(z;v)$ the hyperbolic norm of $v\in \C$ at $z\in \D$, namely,
\[
\kappa_D(z;v):=\frac{|v|}{1-|z|^2}.
\]
If $\gamma:[0,1]\to \D$ is a Lipschitz continuous curve, the {\sl hyperbolic length} of $\gamma$ is
\[
\ell_\D(\gamma):=\int_0^1\kappa_\D(\gamma(t);\gamma'(t))dt.
\]
The integrated distance, {\sl i.e.}, the {\sl hyperbolic distance} in $\D$ is denoted by $\omega$, namely,
\[
\omega(z,w)=\inf_{\gamma}\ell_\D(\gamma),
\]
where $\gamma$ is any Lipschitz continuous curve joining $z$ and $w$. It is well known that $\omega(z,w)=\frac{1}{2}\log \frac{1+|T_z(w)|}{1-|T_z(w)|}$, where $T_z(w)=\frac{z-w}{1-\overline{z}w}$ is an automorphism of $\D$. 

If $\Omega\subsetneq \C$ is a simply connected domain and $z\in \Omega$, $v\in \C$, given a Riemann map $f:\D \to \Omega$, we let 
\[
\kappa_\Omega(z;v):=\kappa_\D(f^{-1}(z); \frac{v}{f'(f^{-1}(z))}).
\]
Similarly, we define the hyperbolic length $\ell_\Omega$ of a curve and the hyperbolic distance $k_\Omega$ between points of $\Omega$. By Schwarz's Lemma, all these hyperbolic quantities are invariant under biholomorphisms and are decreasing under the action of holomorphic functions.

A {\sl geodesics} for the hyperbolic distance is a smooth curve such that the hyperbolic length among any two points of the curve coincide with the hyperbolic distance between the two points. Using the conformal invariance of the hyperbolic distance, it follows studying the case of the unit disc that for every two points there exists a unique (up to parameterization) geodesic joining the two points.

Let $\Ha:=\{w\in \C: \Re w>0\}$ be the right half plane.

Since $\Ha$ is biholomorphic to $\D$ via a Cayley transform $z\mapsto (1+z)/(1-z)$, one can easily prove that
\[
k_{\Ha}(w_1,w_2)=\frac{1}{2}\log \frac{1+\left\vert \frac{w_1-w_2}{w_1+\overline{w_2}}\right\vert}{1-\left\vert \frac{w_1-w_2}{w_1+\overline{w_2}}\right\vert}, \quad w_1,w_2\in \Ha,
\]
and
\begin{equation}\label{Eq:metric-semiplano}
\kappa_\Ha(w;v)=\frac{|v|}{2\Re w}, \quad w\in \Ha, \ v\in \C.
\end{equation}

Moreover, one can easily see that  both lines parallel to the real axis, and arcs of circles orthogonal to the imaginary axis are geodesics in $\Ha$.

Finally, using Carath\'eodory's prime-ends topology (see, {\sl e.g.}, \cite{ColLohbook66}), one can see that for any $z_0\in \Omega$ and any prime end $\underline{x}\in \partial_C\Omega$ (here $\partial_C\Omega$ denotes the set of prime-ends of $\Omega$ endowed with the Carath\'eodory topology), there exists a unique geodesic $\gamma:[0,+\infty)\to \Omega$, parametrized by hyperbolic arc length, so that $\gamma(0)=z_0$ and $\gamma(t)$ converges to $\underline{x}$ in the Carath\'eodory topology. Indeed, this is true in $\oD$ with the Euclidean topology, and since Riemann mappings are isometries for the hyperbolic distance and homeomorphisms for the Carath\'eodory topology and $\oD$ is homeomorphic to $\D\cup \partial_C\D$ endowed with the Carath\'eodory topology, the result follows at once.

The following lemma is a straightforward computation from the very definition:

\begin{lemma}\label{Lem:hyper-semipiano}
Let $\beta\in (-\frac{\pi}{2},\frac{\pi}{2})$. 
\begin{enumerate}
\item Let  $0<\rho_0<\rho_1$ and let $\Gamma:=\{\rho e^{i\beta}: \rho_0\leq \rho\leq \rho_1\}$. Then, $\displaystyle{\ell_{\Ha}(\Gamma)=\frac{1}{2\cos \beta}\log\frac{\rho_1}{\rho_0}}$.
In particular, $\displaystyle{k_\Ha(\rho_0, \rho_1)=\frac{1}{2}\log\frac{\rho_1}{\rho_0}}$.
\item Let $\rho_0, \rho_1>0$. Then, $\displaystyle{k_{\Ha}(\rho_0,\rho_1e^{i\beta})-k_{\Ha}(\rho_0,\rho_1)\geq \frac{1}{2}\log\frac{1}{\cos \beta}.}$
\item Let $\rho_0>0$ and $\alpha\in (-\frac{\pi}{2},\frac{\pi}{2})$. Then, $(0,+\infty)\ni \rho\mapsto k_{\Ha}(\rho e^{i\alpha},\rho_0e^{i\beta})$ has a minimum at $\rho=\rho_0$, it is  increasing for $\rho>\rho_0$ and  decreasing for $\rho<\rho_0$.
\item Let $\theta_0, \theta_1\in (-\frac{\pi}{2},\frac{\pi}{2})$ and $\rho>0$. Then $k_{\Ha}(\rho e^{i\theta_0},\rho e^{i\theta_1})=k_{\Ha}(e^{i\theta_0},e^{i\theta_1})$. Moreover, $k_\Ha(1,e^{i\theta})=k_\Ha(1,e^{-i\theta})$ for all $\theta\in [0,\pi/2)$ and $[0, \pi/2)\ni \theta\mapsto k_\Ha(1,e^{i\theta})$ is strictly increasing.
\item Let $\beta_0,\beta_1\in (-\frac{\pi}{2},\frac{\pi}{2})$ and $0<\rho_0<\rho_1$. Then $\displaystyle{
k_{\Ha}(\rho_0e^{i\beta_0},\rho_1e^{i\beta_1})\geq k_{\Ha}(\rho_0,\rho_1)}$.
\item For all $\rho>0$  we have $\displaystyle{k_\Ha(\rho, \rho e^{i\beta})\leq \frac{1}{2}\log\frac{1}{\cos\beta}+\frac{1}{2}\log 2}$.
\end{enumerate}
\end{lemma}

\section{Hyperbolic projections, tangential and orthogonal speeds of curves in the disc}

In what follows, for not burdening the notation, we will consider geodesics parameterized by (hyperbolic) arc length, but, as it will be clear, this is not relevant, and any parametrization of  geodesics would work as well.

\begin{definition}
Let $\Omega\subsetneq \C$ be a simply connected domain. Let $\gamma:\R\to \Omega$ be a geodesic parameterized by arc length. Let $z\in \Omega$. The {\sl hyperbolic projection $\pi_\gamma(z)\in \gamma(\R)$ of $z$ onto $\gamma$} is the closest point (in the hyperbolic distance) of $\gamma$ to $z$, namely,
\[
k_\Omega(\pi_\gamma(z), z)=\min_{t\in \R}k_\Omega(\gamma(t), z).
\]
\end{definition}

Using conformal invariance, one can easily prove the following:

\begin{proposition}\label{Prop:hyperbolic-projection-geo}
Let $\Omega\subsetneq \C$ be a simply connected domain. Let $\gamma:\R\to \Omega$ be a geodesic in $\Omega$ parameterized by arc length and let $z\in \Omega$. Then $\pi_\gamma(z)$ is the point of intersection of $\gamma$ with the geodesic $\tilde\gamma$ containing $z$ and intersecting $\gamma$ orthogonally (in the Euclidean sense).
\end{proposition}

In particular, by Lemma \ref{Lem:hyper-semipiano}(3), if $\rho e^{i\theta}\in \Ha$, $\rho>0$ and $\theta\in (-\pi/2,\pi/2)$ and $\gamma$ denotes the geodesic given by $\gamma(r)=r$, $r>0$, then
\[
\pi_\gamma(\rho e^{i\theta})=\rho.
\]

Although orthogonal projections onto geodesics are not holomorphic maps, they do not increase the hyperbolic distance:

\begin{proposition}\label{Prop:projection-decrease}
Let  $\Omega\subsetneq \C$ be a simply connected domain, $\gamma:\R \to \Omega$  a geodesic parameterized by arc length. Then for every $z,w \in \Omega$, we have
\[
k_\Omega(\pi_\gamma(z), \pi_\gamma(w))\leq k_\Omega(z,w).
\]
\end{proposition}
\begin{proof}
Since the statement is invariant under isometries for the hyperbolic distance, using a univalent map, we can assume $\Omega=\Ha$ and the image of $\gamma$ is $(0,+\infty)$. We can write $z=\rho_0 e^{i\beta_0}$ with $\rho_0>0$ and $\beta_0\in (-\pi/2,\pi/2)$ and $w=\rho_1 e^{i\beta_1}$ with $\rho_1>0$ and $\beta_1\in (-\pi/2,\pi/2)$. By Lemma \ref{Lem:hyper-semipiano}(3), $\pi_\gamma(z)=\pi_\gamma(\rho_0 e^{i\beta_0})=\rho_0$ and $\pi_\gamma(w)=\pi_\gamma(\rho_1 e^{i\beta_1})=\rho_1$. Hence the result follows immediately from Lemma \ref{Lem:hyper-semipiano}(5). 
\end{proof}

Let $P, Q\in \R^2$ two distinct points, and $R$ any line through $P$---note that a line is a geodesic for the Euclidean metric. Let $\pi_R(Q)$ denote the (Euclidean) orthogonal projection of $Q$ onto $R$. By Pythagoras' Theorem, $|P-\pi_R(Q)|^2+|Q-\pi_R(Q)|^2=|P-Q|^2$. The next result tells that, in hyperbolic geometry, a Pythagoras' Theorem is true up to a universal constant without squaring the distances:

\begin{proposition}[Pytaghoras' Theorem in hyperbolic geometry] \label{Prop:distance-como-suma}
Let  $\Omega\subsetneq \C$ be a simply connected domain, $\gamma:\R \to \Omega$  a geodesic parameterized by arc length, $x_0\in \gamma$ and  $z\in \Omega$. Then
\[
k_\Omega(x_0, \pi_\gamma(z))+k_\Omega(z, \gamma)-\frac{1}{2}\log 2 \leq k_\Omega(x_0, z)\leq k_\Omega(x_0, \pi_\gamma(z))+k_\Omega(z, \gamma),
\]
where $k_\Omega(z, \gamma):=\inf_{t\in \R}k_\Omega(z, \gamma(t))=k_\Omega(z,\pi_\gamma(z))$.
\end{proposition}
\begin{proof}
Since the statement is invariant under isometries for the hyperbolic distance, using a univalent map, we can transfer our considerations to $\Ha$, and we can assume that $\gamma(\R)=(0,+\infty)$ and $x_0=1$.

Let $z\in \Ha$, and write $z=\rho e^{i\beta}$ with $\rho>0$ and $\beta\in (-\pi/2,\pi/2)$. By Lemma \ref{Lem:hyper-semipiano}(3), $\pi_\gamma(\rho e^{i\beta})=\rho$. Hence, by the triangle inequality,
\[
k_\Ha(1,\rho e^{i\beta})\leq k_\Ha(1, \rho)+k_\Ha(\rho, \rho e^{i\beta})=k_\Ha(1, \pi_\gamma(\rho e^{i\beta}))+k_\Ha(\gamma, \rho e^{i\beta}).
\]
On the other hand, by Lemma \ref{Lem:hyper-semipiano}(2),
\begin{equation*}
k_\Ha(1,\rho e^{i\beta})\geq k_\Ha(1,\rho)+\frac{1}{2}\log \frac{1}{\cos \beta}.
\end{equation*}
The previous equation, together with Lemma \ref{Lem:hyper-semipiano}(6), gives
\begin{equation*}
\begin{split}
k_\Ha(1,\rho e^{i\beta})&\geq k_\Ha(1,\rho)+\frac{1}{2}\log \frac{1}{\cos \beta}\geq k_\Ha(1,\rho)+k_\Ha(\rho, \rho e^{i\beta})-\frac{1}{2}\log 2\\& =k_\Ha(1, \pi_\gamma(\rho e^{i\beta}))+k_\Ha(\gamma, \rho e^{i\beta})-\frac{1}{2}\log 2,
\end{split}
\end{equation*}
and we are done.
\end{proof}

The previous proposition allows to make sense to the following definition and the subsequent remarks.

\begin{definition}\label{Def:speed-tan-ort}
Let $\Omega\subsetneq \C$ be a simply connected domain and let $z_0\in \Omega$.
Let $\eta:[0,+\infty)\to \Omega$ be a continuous curve such that $\eta(t)$ converges in the Carath\'eodory topology of $\Omega$ to a prime end $\underline{x}\in \partial_C\Omega$ as $t\to+\infty$. Let $\gamma:(-\infty,+\infty)\to \Omega$ be the geodesic of $\Omega$ parameterized by arc length such that $\gamma(0)=z_0$ and $\gamma(t)\to \underline{x}$ in the Carath\'eodory topology of $\Omega$  as $t\to+\infty$. The {\sl orthogonal speed} of $\eta$ is
\[
v_{\Omega,z_0}^o(\eta;t):=k_\Omega(z_0, \pi_{\gamma}(\eta(t))).
\]
The {\sl tangential speed} $v_{\Omega,z_0}^T(\eta;t)$  of $\eta$ is
\[
v_{\Omega,z_0}^T(\eta;t):=k_\Omega(\gamma, \eta(t)).
\]
\end{definition}

\begin{remark}\label{Rem:speed-inv-conforme}
Let $\Omega$, $z_0$, $\underline{x}$, $\gamma$ and $\eta$ be as in Definition \ref{Def:speed-tan-ort}.
\begin{enumerate}
\item The orthogonal speed and the tangential speed of a curve do not depend on the parameterization of the geodesic $\gamma$. Therefore,  the definition of orthogonal speed and tangential speed depend only on $\Omega, z_0$ and $\underline{x}$.
\item If $\Omega, \Omega'\subsetneq \C$ are simply connected domains, $z_0\in \Omega$, $z_0'\in \Omega'$ and $f:\Omega\to \Omega'$ is a biholomorphism such that $f(z_0)=z_0'$ then $v_{\Omega,z_0}^o(\eta;t)=v_{\Omega',z'_0}^o(f\circ \eta; t)$ and $v_{\Omega,z_0}^T(\eta;t)=v_{\Omega',z'_0}^T(f\circ \eta; t)$ for all $t\geq 0$. This follows immediately since $f$ is an isometry for the hyperbolic distances of $\Omega$ and $\Omega'$.
\end{enumerate}
\end{remark}

The actual orthogonal speed and tangential speed of a curve depend on the base point chosen, but, asymptotically they do not:

\begin{lemma}\label{Lem:speed-tan-ort-base}
Let $\Omega\subsetneq \C$ be a simply connected domain and let $z_0, z_1\in \Omega$. Then for every $\underline{x}\in \partial_C\Omega$ and for every continuous curve $\eta:[0,+\infty)\to \Omega$ converging to $\underline{x}$ in the Carath\'eodory topology of $\Omega$, we have
\begin{enumerate}
\item $\lim_{t\to+\infty} v_{\Omega,z_0}^o(\eta;t)=+\infty$,
\item $\lim_{t\to+\infty} |v_{\Omega,z_0}^T(\eta;t)-v_{\Omega,z_1}^T(\eta;t)|=0$,
\item $\limsup_{t\to+\infty}|v_{\Omega,z_0}^o(\eta;t)-v_{\Omega,z_1}^o(\eta;t)|\leq k_\Omega(z_0,z_1)$.
\end{enumerate}
\end{lemma}

\begin{proof}
By Remark \ref{Rem:speed-inv-conforme}(2), up to composing with a biholomorphism from $\Ha$ to $\Omega$, we can assume $\Omega=\Ha$, $z_0=1$ and $\underline{x}$ is the prime end of $\Ha$ which corresponds to ``$\infty$'', namely, the prime end defined by the null chain $\{(n+1)e^{i\theta}: |\theta|<\pi/2\}_{n\in \N}$.  Hence, $\lim_{t\to+\infty}|\eta(t)|=+\infty$. Moreover, the geodesic in $\Ha$ which joins $1$ to $\underline{x}$ is $\gamma_0(r):=r$, $r\in (0,+\infty)$. While, the geodesic in $\Ha$ which joins $z_1:=x+iy$ to $\underline{x}$ is $\gamma_1(r):=r+iy$, $r\in (0,+\infty)$.

From Lemma \ref{Lem:hyper-semipiano}(3), we have $\pi_{\gamma_0}(\eta(t))=|\eta(t)|$. This shows in particular that
\[
v_{\Ha,1}^o(\eta;t)=k_\Ha(1, \pi_{\gamma_0}(\eta(t)))=k_\Ha(1, |\eta(t)|)\to +\infty,
\]
as $t\to +\infty$, and (1) follows.

On the other hand, using the automorphism $z\mapsto z-iy$ which maps $\gamma_0$ onto $\gamma_1$ and taking into account that it is an isometry for $k_\Ha$, we see that $\pi_{\gamma_1}(\eta(t))=|\eta(t)-iy|+iy$.

Therefore,
\begin{equation*}
\begin{split}
|v_{\Ha,1}^T(\eta;t)-v_{\Ha,x+iy}^T(\eta;t)|&=|k_\Ha(\eta(t), \pi_{\gamma_0}(\eta(t)))-k_\Ha(\eta(t), \pi_{\gamma_1}(\eta(t)))|\\&\leq k_\Ha(\pi_{\gamma_0}(\eta(t)), \pi_{\gamma_1}(\eta(t)))=k_\Ha(|\eta(t)|, |\eta(t)-iy|+iy).
\end{split}
\end{equation*}
Taking into account that $\lim_{t\to+\infty}|\eta(t)|=+\infty$, a direct computation shows that
\begin{equation}\label{Eq:zero-t-ort-geo}
\lim_{t\to+\infty}k_\Ha(|\eta(t)|, |\eta(t)-iy|+iy)=0,
\end{equation}
 and hence (2) follows.

Now, using the triangle inequality,
\begin{equation*}
\begin{split}
|v_{\Ha,1}^o(\eta;t)-v_{\Ha,x+iy}^o(\eta;t)|&=|k_\Ha(1,\pi_{\gamma_0}(\eta(t)))-k_\Ha(x+iy, \pi_{\gamma_1}(\eta(t)))|\\&=|k_\Ha(1,\pi_{\gamma_0}(\eta(t)))-k_\Ha(x+iy,\pi_{\gamma_0}(\eta(t)))\\&+k_\Ha(x+iy,\pi_{\gamma_0}(\eta(t)))-k_\Ha(x+iy, \pi_{\gamma_1}(\eta(t)))|\\&\leq k_\Ha(1,x+iy)+k_\Ha(\pi_{\gamma_0}(\eta(t)),\pi_{\gamma_1}(\eta(t)))\\&= k_\Ha(1,x+iy)+k_\Ha(|\eta(t)|, |\eta(t)-iy|+iy),
\end{split}
\end{equation*}
and thus (3) follows from \eqref{Eq:zero-t-ort-geo}.
\end{proof}

The reason for the name ``tangential speed'' follows from the following property:

\begin{proposition}\label{Prop:expre-speed-in-disc}
Let $\eta:[0,+\infty)\to \D$ be a continuous curve converging to a point $\sigma\in \partial \D$. Let
\[
t_0:=\inf\{s\geq 0: \Re(\overline{\sigma}\eta(t))\geq 0 \ \forall t\in [s,+\infty)\}.
\]
Then $t_0\in [0,+\infty)$ and  for all $t\geq t_0$,
\begin{equation*}
\begin{split}
&\left|\omega(0,\eta(t))-\frac{1}{2}\log\frac{1}{1-|\eta(t)|}\right|\leq \frac{1}{2}\log 2,\\
&\left|v^o_{\D,0}(\eta;t)-\frac{1}{2}\log \frac{1}{|\sigma-\eta(t)|} \right|\leq \frac{1}{2}\log 2,\\
&\left|v^T_{\D,0}(\eta;t)-\frac{1}{2}\log \frac{|\sigma-\eta(t)|}{1-|\eta(t)|} \right|\leq \frac{3}{2}\log 2.
\end{split}
\end{equation*}
\end{proposition}
\begin{proof}
Since $\eta(t)\to \sigma$ as $t\to +\infty$, it follows that $t_0<+\infty$.

The first equation follows immediately from the very definition of $\omega$. Indeed, for every $t\geq 0$,
\[
\left|\omega(0,\eta(t))-\frac{1}{2}\log\frac{1}{1-|\eta(t)|}\right|=\frac{1}{2}\log (1+|\eta(t)|)<\frac{1}{2}\log 2.
\]

In order to prove the other two equations, up to change $\eta$ with $\overline{\sigma}\eta$, we can assume without loss of generality that $\sigma=1$. Let $C:\D\to \Ha$ be the Cayley transform given by $C(z)=\frac{1+z}{1-z}$. For every $t\geq 0$, let us write $\rho_te^{i\theta_t}:=C(\eta(t))$, with $\rho_t>0$ and $\theta_t\in (-\pi/2,\pi/2)$.  This implies in particular, that $\rho_t\geq 1$ for all $t\geq t_0$. Then, for $t\geq t_0$ we have
\begin{equation}\label{Eq:scrivo-in-Ha}
\begin{split}
v^o_{\D,0}(\eta;t)&=v^o_{\Ha, 1}(\rho_t e^{i\theta_t};t)=k_\Ha(1, \rho_t)=\frac{1}{2} \log \rho_t\\&=\frac{1}{2} \log|C(\eta(t))º| =\frac{1}{2} \log \frac{|1+\eta(t)|}{|1-\eta(t)|},
\end{split}
\end{equation}
where, the first equality follows from Remark \ref{Rem:speed-inv-conforme}(2), the second equality follows from the definition of orthogonal speed and since the orthogonal projection of $\rho_te^{i\theta_t}$ onto the geodesic $(0,+\infty)$ is $\rho_t$ by Lemma \ref{Lem:hyper-semipiano}(3) and the third equality follows from Lemma \ref{Lem:hyper-semipiano}(1).

Therefore, by \eqref{Eq:scrivo-in-Ha}, and taking into account that for $t\geq t_0$ we have $|1+\eta(t)|\geq 1+\Re\eta(t)\geq 1$,
\[
\left|v^o_{\D,0}(\eta;t)-\frac{1}{2} \log \frac{1}{|1-\eta(t)|}\right|=\frac{1}{2}\log |1+\eta(t)|\leq \frac{1}{2}\log 2.
\]
As for the last inequality, from Proposition \ref{Prop:distance-como-suma} we have
\[
\omega(0,\eta(t))-v^o_{\D, 0}(\eta;t)\leq v^T_{\D, 0}(\eta;t)\leq \omega(0,\eta(t))-v^o_{\D, 0}(\eta;t)+\frac{1}{2}\log 2,
\]
and using the previous two inequalities for the estimates of $\omega(0,\eta(t))$ and $v^o_{\D, 0}(\eta;t)$, we get the result.
\end{proof}

\begin{remark}\label{Rem:tang-speed-conv-tg}
As a consequence of the previous proposition, we have that if $\eta:[0,+\infty)\to \D$ is a continuous curve such that $\lim_{t\to+\infty}\eta(t)=\sigma\in \partial \D$, then $\eta$ converges to $\sigma$ non-tangentially if and only if $\limsup_{t\to+\infty}v^T_{\D,0}(\eta;t)<+\infty$.
\end{remark}

\section{Continuous non-elliptic semigroups of holomorphic self-maps of the unit disc}

In this paper we consider only non-elliptic (continuous) semigroups of holomorphic self-maps of the unit disc. We refer the reader to, {\sl e.g.}, \cite{Ababook89, A1, A2, A3, A4, A5, A6, A7, BerPor78, BCD, BCD2, BCGD, ElKhReSh10b2, EliShobook10, ElShZa08, Shobook01, Sis85, Sis98} for all unproved statements and more on the subject.

A {\sl continuous non-elliptic semigroups of holomorphic self-maps of the unit disc}, or just a {\sl non-elliptic semigroup} for short, is a family $(\phi_t)$ such that for every $t\geq 0$, $\phi_t:\D \to \D$ is holomorphic, with no fixed point in $\D$ for $t>0$, $\phi_{t+s}=\phi_t \circ \phi_s$ for all $t,s\geq 0$, $\phi_0(z)=z$ for all $z\in \D$ and $[0,+\infty)\ni t\mapsto \phi_t$ is continuous with respect to the topology of uniform convergence on compacta of $\D$.

If $(\phi_t)$ is a non-elliptic semigroup in $\D$, there exists a point $\tau\in \partial \D$, the {\sl Denjoy-Wolff point} of $(\phi_t)$ such that $\lim_{t\to\infty}\phi_t(z)=\tau$ for all $z\in \D$, and the convergence is uniform on compacta.

Moreover, the angular derivative $\phi'_t(\tau)$ of $\phi_t$ at $\tau$ exists for all $t\geq 0$  and there exists $\lambda\geq 0$, {\sl the spectral value of $(\phi_t)$} such that 
\[
\phi'_t(\tau)=e^{-\lambda t}
\]
 for all $t\geq 0$.

If $(\phi_t)$ is a semigroup in $\D$, there exists an (essentially unique) {\sl holomorphic model} $(\Omega, h, z+it)$, where $h:\D\to \C$ is univalent, $h(\D)$ is starlike at infinity (namely, $h(\D)+it\subseteq h(\D)$ for all $t\geq 0$) and $h(\phi_t(z))=h(z)+it$ for all $z\in \D$ and $t\geq 0$. Moreover, $\Omega=\bigcup_{t\geq 0}h(\D)-it$ and we have the following cases: $\Omega$ is either a strip $\strip_r:=\{z\in \C: 0<\Re z<r\}$ (where $r=\pi/\lambda$ with $\lambda>0$ the spectral value of $(\phi_t)$),  or the right half plane $\Ha$, or the left half plane $\Ha^-:=\{w\in \C: \Re w<0\}$ or $\C$.  The holomorphic model is universal in the sense that any other (semi)conjugation of $(\phi_t)$ factorizes through it (see \cite{Cow81, AroBra16}). The map $h$ is called the {\sl Koenigs function} of $(\phi_t)$.

The semigroup is {\sl hyperbolic} if $\Omega$ is a strip, it is {\sl parabolic} otherwise. Moreover, parabolic semigroups are {\sl of finite hyperbolic step} if $\Omega$ is a half plane, or {\sl of zero hyperbolic step} if $\Omega=\C$. 

This definition is equivalent to the classical one, for which a semigroup $(\phi_t)$ is hyperbolic provided its spectral value is $>0$,  it is parabolic if its spectral value is $0$, and the hyperbolic step is positive if $\lim_{t\to \infty}\omega(\phi_t(z), \phi_{t+1}(z))>0$ for some---and hence any---$z\in \D$. The last  equivalence follows from the fact that $k_{\Omega}(z,w)=\lim_{t\to \infty}\omega(\phi_t(z), \phi_t(w))$ (see \cite{AroBra16}). 

\section{Speeds of non-elliptic semigroups}

Since the orbits of a non-elliptic semigroup converge to the Denjoy-Wolff point on $\partial \D$, one might study the tangential and orthogonal speed of convergence. First of all, we show that the (asymptotic behavior of) orthogonal speed and the tangential speed of an orbit of a semigroup do not depend on the starting point:

\begin{lemma}\label{Lem:speed-indip-orbit}
Let $(\phi_t)$ be a non-elliptic semigroup in $\D$ with Denjoy-Wolff point $\tau\in \partial \D$. Let $z_1, z_2\in \D$ and let $\eta_j:[0,+\infty)\to \D$ be the continuous curve defined by $\eta_j(t):=\phi_t(z_j)$, $j=1,2$. Then for every $t\geq 0$
\[
|v^o_{\D, 0}(\eta_1;t)-v^o_{\D, 0}(\eta_2;t)|\leq \omega(z_1,z_2),
\]
\[
|v^T_{\D, 0}(\eta_1;t)-v^T_{\D, 0}(\eta_2;t)|\leq 2\omega(z_1,z_2).
\]
\end{lemma}
\begin{proof}
Let $\gamma:(-1,1)\to \D$ be the geodesic of $\D$ defined by $\gamma(r)=r\tau$. For $z\in \D$ let $\pi_\gamma(z)$ be the orthogonal projection of $z$ onto $\gamma$. Then, by the very definition of orthogonal speed of curves and Proposition \ref{Prop:projection-decrease}, we have
\begin{equation*}
\begin{split}
|v^o_{\D, 0}(\eta_1;t)-v^o_{\D, 0}(\eta_2;t)|&=|\omega(0, \pi_\gamma(\eta_1(t)))-\omega(0, \pi_\gamma(\eta_2(t)))|\\&\leq \omega(\pi_\gamma(\eta_1(t)), \pi_\gamma(\eta_2(t)))\leq \omega(\eta_1(t), \eta_2(t))\\&=\omega(\phi_t(z_1), \phi_t(z_2))\leq \omega(z_1,z_2).
\end{split}
\end{equation*}
A similar argument proves the second inequality. Namely, 
\begin{equation*}
\begin{split}
v^T_{\D, 0}(\eta_1;t)&=\omega (\phi_t(z_1), \pi_\gamma( \phi_t(z_1)))\\
&\leq  \omega (\phi_t(z_1),  \phi_t(z_2))+\omega (\phi_t(z_2), \pi_\gamma( \phi_t(z_2)))+\omega (\pi_\gamma(\phi_t(z_2)), \pi_\gamma( \phi_t(z_1))) \\
&=\omega (\phi_t(z_1),  \phi_t(z_2))+v^T_{\D, 0}(\eta_2;t)+\omega (\pi_\gamma(\phi_t(z_2)), \pi_\gamma( \phi_t(z_1)))\\
&\leq 2\omega(z_1,z_2)+v^T_{\D, 0}(\eta_2;t).
\end{split}
\end{equation*}
That is, $v^T_{\D, 0}(\eta_1;t)-v^T_{\D, 0}(\eta_2;t)\leq 2\omega (z_1,z_2)$. Changing the role of $z_1$ and $z_2$, we obtain the second inequality of the statement.
\end{proof}

Lemmas \ref{Lem:speed-indip-orbit} and \ref{Lem:speed-tan-ort-base} show that, in order to study the asymptotic behavior of the speed of convergence of semigroups' orbits to the Denjoy-Wolff point, it is enough to study the orbit starting at $0$ and considering the speed with respect to $0$. In other words, the following definition makes sense:

\begin{definition}
Let $(\phi_t)$ be a non-elliptic semigroup in $\D$ with Denjoy-Wolff point $\tau\in \partial \D$. For $t\geq 0$, we let
\[
v(t):=\omega(0,\phi_t(0)),
\]
and call $v(t)$ the {\sl total speed of $(\phi_t)$}. 

Also, let $\gamma:(-1,1)\to \D$ be the geodesic of $\D$ defined by $\gamma(r):=r\tau$ and let $\pi_\gamma:\D \to \gamma((-1,1))$ be the orthogonal projection.  For $t\geq 0$, we let
\[
v^o(t):=v^o_{\D,0}(\phi_t(0);t):=\omega(0,\pi_\gamma(\phi_t(0))),
\]
and call $v^o(t)$ the  {\sl orthogonal speed of $(\phi_t)$.} Finally, we let 
\[
v^T(t):=v^T_{\D,0}(\phi_t(0);t):=\omega(\phi_t(0),\pi_\gamma(\phi_t(0))),
\]
and call $v^T(t)$ the  {\sl tangential speed of $(\phi_t)$.}
\end{definition}

\begin{remark}\label{Rem:uno-piu-conv-tg}
It follows immediately from Remark \ref{Rem:tang-speed-conv-tg} that the orbit $[0,+\infty)\ni t\mapsto \phi_t(z)$ converges non-tangentially to $\tau$ for some---and hence any---$z\in \D$ if and only if $\limsup_{t\to+\infty}v^T(t)<+\infty$.
\end{remark}

It follows from Lemma \ref{Lem:hyper-semipiano} and the previous considerations that, if $(\phi_t)$ is a non-elliptic semigroup in $\D$ with Denjoy-Wolff point $\tau\in \partial \D$, and $C(z)=(\tau+z)/(\tau-z)$ (a biholomorphism from $\D$ to $\Ha$), setting $\rho_t e^{i\theta_t}=C(\phi_t(C^{-1}(1)))$ with $\rho_t>0$ and $\theta_t\in (-\pi/2,\pi/2)$, then
\begin{equation}\label{Eq:speed-in-H}
v^o(t)\sim \frac{1}{2}\log \rho_t,\quad  v^T(t)\sim \frac{1}{2}\log \cos \frac{1}{\theta_t}.
\end{equation}

By Proposition \ref{Prop:distance-como-suma}, if $(\phi_t)$ is a non-elliptic semigroup we have
\begin{equation}\label{Eq:split-speed-semig}
v^o(t)+v^T(t)-\frac{1}{2}\log 2\leq v(t)\leq v^o(t)+v^T(t).
\end{equation}

A second less immediate relation between the orthogonal speed and the tangential speed is contained in the following proposition:

\begin{proposition}\label{Prop:first-prop-velocity-orbit}
If $(\phi_t)$ is a non-elliptic semigroup in $\D$ then, for every $t\geq 0$,
\begin{equation}\label{Eq:split-speed-semig2}
v^T(t)\leq v^o(t)+4\log 2.
\end{equation}
\end{proposition}
\begin{proof}
Let $\tau\in \partial \D$ be the Denjoy-Wolff point of $(\phi_t)$ and let $\lambda\geq 0$ be its spectral value.
 By the Julia's Lemma, for every $t\geq 0$
\[
\frac{|\tau-\phi_t(0)|}{1-|\phi_t(0)|}\leq 4\sqrt{\frac{e^{-\lambda t}}{1-|\phi_t(0)|^2}},
\]
which is equivalent to
\[
e^{\lambda t}\frac{1+|\phi_t(0)|}{1-|\phi_t(0)|}\leq \frac{16}{|\tau-\phi_t(0)|^2}.
\]
Applying the function $x\mapsto \frac{1}{2}\log x$ to the previous inequality, we obtain for every $t\geq 0$,
\begin{equation*}
\begin{split}
\frac{1}{2}\log \frac{1}{1-|\phi_t(0)|}&\leq \frac{\lambda t}{2}+\frac{1}{2}\log \frac{1}{1-|\phi_t(0)|}+\frac{1}{2}\log(1+|\phi_t(0)|)\\&\leq \frac{1}{2}\log 16 +\log \frac{1}{|\tau-\phi_t(0)|}.
\end{split}
\end{equation*}
 Therefore, by Proposition \ref{Prop:expre-speed-in-disc}, we have for all $t\geq 0$,
\begin{equation*}
\begin{split}
v(t)&\leq \frac{1}{2}\log \frac{1}{1-|\phi_t(0)|}+\frac{1}{2}\log 2\\&\leq \frac{1}{2}\log 16 +\log \frac{1}{|\tau-\phi_t(0)|}+\frac{1}{2}\log 2\\&\leq \frac{1}{2}\log 16+\frac{3}{2}\log 2+2v^o(t)=2v^o(t)+\frac{7}{2}\log 2.
\end{split}
\end{equation*}
Hence, by \eqref{Eq:split-speed-semig}, we have for all $t\geq 0$,
\[
v^o(t)+v^T(t)\leq v(t)+\frac{1}{2}\log 2\leq 2v^o(t)+\frac{7}{2}\log 2+\frac{1}{2}\log 2.
\]
Finally, the previous equation implies that $v^T(t)\leq v^o(t)+4\log 2$ for all $t\geq 0$, and we are done.
\end{proof}

The speeds of convergence are essentially invariant under conjugation:

\begin{proposition}\label{Prop:conjug-speed}
Let $(\phi_t)$ and $(\psi_t)$ be two non-elliptic semigroups in $\D$. Suppose there exists $M\in \Aut(\D)$ such that $\phi_t=M^{-1}\circ \psi_t\circ M$ for all $t\geq 0$. Denote by $v(t), v^o(t), v^T(t)$ ({\sl respectively}, $\tilde v(t), \tilde v^o(t), \tilde v^T(t)$) the total speed, orthogonal speed and tangential speed of $(\phi_t)$ ({\sl respect.} of $(\psi_t)$). Then there exists $C>0$ such that for all $t\geq 0$
\begin{equation*}
\begin{split}
&|v(t)-\tilde v(t)|<C,\\
&|v^o(t)-\tilde v^o(t)|<C,\\
&|v^T(t)-\tilde v^T(t)|<C.
\end{split}
\end{equation*}
\end{proposition}
\begin{proof}
Let $\tau\in \partial \D$ be the Denjoy-Wolff point of $(\phi_t)$ and $\tilde \tau\in \partial \D$ that of $(\psi_t)$. Let $\gamma:(0,+\infty)\to \D$ ({\sl respectively}, $\tilde \gamma:(0,+\infty)\to \D$) be the geodesic in $\D$ parameterized by arc length such that $\gamma(0)=0$ ({\sl respect.}, $\tilde\gamma(0)=0$) and $\lim_{t\to+\infty}\gamma(t)=\tau$ ({\sl respect.},  $\lim_{t\to+\infty}\gamma(t)=\tilde\tau$). 

Since $M$ is an isometry for the hyperbolic distance, for all $t\geq 0$,
\[
v(t)=\omega(0, \phi_t(0))=\omega(0, (M^{-1}\circ \psi_t\circ M)(0))=\omega(M(0), \psi_t(M(0)). 
\]
Hence, for all $t\geq 0$,
\begin{equation*}
\begin{split}
|v(t)-\tilde v(t)|&=|\omega(M(0), \psi_t(M(0))-\omega(0, \psi_t(0))|\\ &\leq  |\omega(M(0), \psi_t(M(0))-\omega(0, \psi_t(M(0))|\\&+|\omega(0, \psi_t(M(0))-\omega(0, \psi_t(0))|\\&\leq \omega(M(0), 0)+\omega(\psi_t(M(0)), \psi_t(0))\leq 2\omega(M(0), 0)=:C_0.
\end{split}
\end{equation*}
Moreover, since $M$ is an isometry for the hyperbolic distance, the curve $\gamma_1:(0,+\infty)\to \D$ defined by $\gamma_1:=M^{-1}\circ \gamma$ is a geodesic in $\D$ parameterized by arc length. Hence, for all $t\geq 0$,
\[
v^T(t)=\omega(\phi_t(0), \gamma)=\omega(M^{-1}(\phi_t(0)),  \gamma_1)=\omega(\psi_t(M^{-1}(0)), \gamma_1).
\]
By Lemma \ref{Lem:speed-tan-ort-base}, $\lim_{t\to+\infty}|\tilde v^T(t)-\omega(\psi_t(M^{-1}(0), \gamma_1))|=0$, thus there exists $C_1>0$ such that $|v^T(t)-\tilde v^T(t)|<C_1$ for all $t\geq 0$.

Finally, by \eqref{Eq:split-speed-semig} we have for all $t\geq 0$,
\[
v^o(t)-\tilde v^o(t)\leq v(t)-v^T(t)+\frac{1}{2}\log 2-\tilde v(t)+\tilde v^T(t)\leq  C_0+C_1+\frac{1}{2}\log 2.
\]
The same argument proves that $\tilde v^o(t)-v^o(t)\leq C_0+C_1+\frac{1}{2}\log 2$, and we are done.
\end{proof}

If $\Omega$ is a domain starlike at infinity, and $p\in \Omega$, we let 
\[
\Omega^+:=\Omega\cup \{w\in \C: \Re w>\Re p\}, \quad \Omega^-:=\Omega\cup \{w\in \C: \Re w<\Re p\}.
\]
Note that $\Omega^\pm$ is a domain starlike at infinity. Moreover, for any open set $D\subset \C$ and $p\in D$, we let
\[
\delta_D(p)=\inf\{|z-p|: z\in \C\setminus D\}.
\]

The following result is a consequence of \cite{BCGDZ} and Remark \ref{Rem:uno-piu-conv-tg}:

\begin{theorem}
Let $(\phi_t)$ be a non-elliptic semigroup in $\D$, with Koenigs function $h$. Let $p\in h(\D)$. Then $\limsup_{t\to \infty}v^T(t)<+\infty$ if and only if there exists $C>0$ such that 
\[
\frac{1}{C}\min\{t,\delta_{h(\D)^+}(p+it)\}\leq \min\{t,\delta_{h(\D)^-}(p+it)\}\leq C\min\{t,\delta_{h(\D)^+}(p+it)\}
\]
for all $t\geq 0$.
\end{theorem} 

In particular, if $(\phi_t)$ is hyperbolic, there exists $C>0$ such that $v^T(t)\leq C$ for all $t\geq 0$. Hence, for hyperbolic semigroups, $v^o(t)\sim v(t)$. 

Note that this implies that, in particular, for hyperbolic semigroup the orthogonal speed is {\sl essentially monotone}, in the sense that, if $(\phi_t)$ is a hyperbolic semigroup with Koenigs function $h$, total speed $v(t)$ and orthogonal speed $v^o(t)$ and $(\tilde\phi_t)$ is a hyperbolic semigroup with Koenigs function $\tilde h$ and $h(\D)\subset \tilde h(\D)$, total speed $\tilde v(t)$ and orthogonal speed $\tilde v^o(t)$, then by \eqref{Eq:split-speed-semig}, 
\[
v^o(t)\geq \tilde v^o(t)+C
\]
 for all $t\geq 0$ and some $C>0$, since in the previous case, $v(t)\geq \tilde v(t)$ for all $t\geq 0$ by the monotonicity of the hyperbolic distance.
 
\section{Total speed of convergence}

In this section we consider the total speed of convergence of orbits of hyperbolic and parabolic semigroups to the Denjoy-Wolff point. 

\begin{proposition}\label{Prop:total-speed}
Let $(\phi_t)$ be a non-elliptic semigroup in $\D$, with Denjoy-Wolff point $\tau\in \partial \D$ and $\phi_t'(\tau)=e^{-\lambda t}$ for $\lambda\geq 0$ and $t\geq 0$ (in particular, $(\phi_t)$ is hyperbolic if $\lambda>0$, parabolic otherwise).  Then 
\begin{equation}\label{Eq:rate conv hyp}
\lim_{t\to +\infty}\frac{v(t)}{t}=\lim_{t\to+\infty} \frac{v^{o}(t)}{t}=\frac{\lambda}{2},
\end{equation}
and
\begin{equation*}
\lim_{t\to +\infty}\frac{v^T(t)}{t}=0.
\end{equation*}
\end{proposition}
\begin{proof} 
By \cite{AroBra16}, 
$$
\frac{\lambda}{2}=\lim_{t\to +\infty} \frac{\omega(0,\phi_t(0))}{t}=\lim_{t\to +\infty}\frac{v(t)}{t}.
$$

In case $\lambda=0$, that is, $(\phi_t)$ is parabolic, it follows immediately from \eqref{Eq:split-speed-semig} that
\[
\lim_{t\to+\infty} \frac{v^{o}(t)}{t}=\lim_{t\to +\infty}\frac{v^T(t)}{t}=0.
\]

In case $\lambda>0$, that is, $(\phi_t)$ is hyperbolic, we already noticed that $\limsup_{t\to+\infty}v^T(t)<+\infty$. Thus from \eqref{Eq:split-speed-semig} we have the result. 
\end{proof}

According to the type of the semigroup, we have also a simple  lower bound on the total speed:

\begin{proposition}\label{Prop:total-speed-lowerb}
Let $(\phi_t)$ be a non-elliptic semigroup in $\D$, with Denjoy-Wolff point $\tau\in \partial \D$.  
\begin{itemize}
\item If $(\phi_t)$ is hyperbolic  with   spectral value $\lambda> 0$, then 
\[
\liminf_{t\to+\infty}[v(t)-\frac{\lambda}{2}t]>-\infty,
\]
\item if  $(\phi_t)$ is parabolic of positive hyperbolic step, then 
\[
\liminf_{t\to+\infty}[v(t)-\log t]>-\infty,
\]
\item if  $(\phi_t)$ is parabolic of zero hyperbolic step, then 
\[
\liminf_{t\to+\infty}[v(t)-\frac{1}{4}\log t]>-\infty.
\]
\end{itemize}
\end{proposition}
\begin{proof}
Let $(\phi_t)$ be hyperbolic with spectral value $\lambda>0$. The canonical model of $(\phi_t)$ is $(\strip_{\frac{\pi}{\lambda}}, h, z+it)$. Hence, for every $t\geq 0$,
\begin{equation*}
\begin{split}
v(t)&=\omega(0,\phi_t(0))=k_{h(\D)}(h(0), h(\phi_t(0)))\\&=k_{h(\D)}(h(0), h(0)+it)\geq k_{\strip_{\pi/\lambda}}(h(0), h(0)+it)\\&
\geq k_{\strip_{\pi/\lambda}} (\frac{\pi}{2\lambda},\frac{\pi}{2\lambda}+it)-k_{\strip_{\pi/\lambda}} (\frac{\pi}{2\lambda},h(0))-k_{\strip_{\pi/\lambda}} (h(0)+it,\frac{\pi}{2\lambda}+it)\\
&= \frac{\lambda}{2}t-2k_{\strip_{\pi/\lambda}} (\frac{\pi}{2\lambda},h(0)),
\end{split}
\end{equation*}
where the last equality follows from a direct computation and taking into account that $k_{\strip_{\pi/\lambda}} (h(0)+it,\frac{\pi}{2\lambda}+it)=k_{\strip_{\pi/\lambda}} (h(0),\frac{\pi}{2\lambda})$ for all $t\in \R$ since $z\mapsto z+it$ is an automorphism of $\strip_{\frac{\pi}{\lambda}}$. From this, the result for hyperbolic semigroups follows at once.

Now, assume that $(\phi_t)$ is parabolic of positive hyperbolic step. We can assume that its canonical model is $(\Ha, h, z+it)$ (in case the canonical model is $(\Ha^-, h, z+it)$ the argument is similar). Arguing as in the hyperbolic case, we see that
\[
v(t)\geq k_\Ha(1,1+it)+C,
\]
for some constant $C\in \R$ and every $t\geq 0$. Now, write $1+it=\rho_t e^{i\theta_t}$ for $\rho_t>0$ and $\theta_t\in [0, \pi/2)$. A simple computation shows that $\rho_t=\sqrt{1+t^2}$ and $\cos \theta_t=\frac{1}{\sqrt{1+t^2}}$. Therefore, by Lemma \ref{Lem:hyper-semipiano}(1) and (2), we have
\[
k_\Ha(1,1+it)\geq k_\Ha(1, \sqrt{1+t^2})+\frac{1}{2}\log \sqrt{1+t^2}=\log \sqrt{1+t^2}\geq \log t,
\]
and the result follows in this case as well.

Finally, in case $(\phi_t)$ is parabolic of zero hyperbolic step, the canonical model is $(\C, h, z+it)$. Since $h(\D)$ is starlike at infinity and is different from $\C$, there exists $p\in \C$ such that $p-it\not\in h(\D)$ for all $t\geq 0$ and $p+it\in h(\D)$ for all $t>0$.  Hence, $h(\D)\subseteq \mathcal K_p$, where $\mathcal K_p$ is the Koebe domain $\C\setminus\{\zeta\in \C: \Re \zeta=\Re p, \Im \zeta\leq \Im p\}$. Therefore, arguing as in the previous cases, we find $C\in \R$ such that for every $t\geq 0$,
\[
v(t)\geq k_{\mathcal K_p}(p+i, p+ti)+C=k_{\mathcal K_0}(i, ti)+C. 
\]
Taking into account that the map $\mathcal K_0 \ni z\mapsto \sqrt{-iz}\in \Ha$ is a biholomorphism, where the branch of the square root is chosen so that $\sqrt{1}=1$, we have by Lemma \ref{Lem:hyper-semipiano}(1) 
\[
k_{\mathcal K_0}(i, ti)=k_\Ha(1,\sqrt{t})=\frac{1}{4}\log t,
\]
and we are done.
\end{proof}

\begin{remark}\label{Rem:limit-case-speed}
The bound given by Proposition \ref{Prop:total-speed-lowerb} is sharp. Indeed, as it is clear from the proof, if $(\phi_t)$ is a hyperbolic group in $\D$ with spectral value $\lambda>0$ then there exists $C>0$ such that $|v(t)-\frac{\lambda}{2} t|<C$ for every $t\geq 0$, while, if $(\phi_t)$ is a parabolic group then  there exists $C>0$ such that $|v(t)-\log t|<C$ for every $t\geq 0$---so that, in this sense, non-elliptic groups in $\D$ have the lowest total speed. Moreover, the semigroup $(\phi_t)$ in $\D$ defined as $\phi_t(z):=h^{-1}(h(z)+it)$, $z\in \D$, where $h:\D \to \mathcal K_0$ is a Riemann map for the Koebe domain $\mathcal K_0$, has the property that there exists $C>0$ such that $|v(t)-\frac{1}{4}\log t|<C$ for all $t\geq 0$.
\end{remark}

A direct consequence of Proposition \ref{Prop:total-speed} and Proposition \ref{Prop:total-speed-lowerb} is the following:

\begin{corollary}
Let $(\phi_t)$ be a non-elliptic semigroup in $\D$. Then
\[
\liminf_{t\to+\infty}\frac{v(t)}{\log t}>0, \quad \limsup_{t\to+\infty}\frac{v(t)}{t}<+\infty.
\]
\end{corollary}

As it is clear from the proof of the previous proposition, one can get lower or upper estimates on the total speed of convergence according to the geometry of the image of the Koenigs function using the domain monotonicity of the hyperbolic distance. We provide here an example of such situation by studying a particular case.

For $\alpha, \beta\in [0,\pi]$, with $\alpha+\beta>0$, we denote 
 $$
 V(\alpha,\beta):=\left\{re^{i\theta}:r>0,\;-\alpha<\theta<\beta\right\}.
 $$
 
 \begin{proposition}\label{Prop:precise-est-angles}
Let $(\phi_t)$ be a non-elliptic semigroup in $\D$ with Koenigs function $h$. Suppose $h(\D)=p+iV(\alpha,\beta)$ for some $\alpha, \beta\in (0,\pi]$ with $\alpha+\beta>0$.  
\begin{enumerate}
\item If $\alpha>0, \beta>0$ then there exists $C>0$ such that $v^T(t)\leq C$ and
\[
|v^o(t)-\frac{\pi}{2(\alpha+\beta)}\log t|\leq C, \quad |v(t)-\frac{\pi}{2(\alpha+\beta)}\log t|\leq C,
\]
for all $t\geq 0$.
\item  If either $\alpha=0$ or $\beta=0$ then there exists $C>0$ such that for all $t\geq 0$
\begin{equation*}
\begin{split}
& |v(t)-\frac{\pi+\alpha+\beta}{2(\alpha+\beta)}\log t|\leq C\\
& |v^o(t)-\frac{\pi}{2(\alpha+\beta)}\log t|\leq C\\
& |v^T(t)-\frac{1}{2}\log t|\leq C.
\end{split}
\end{equation*}
\end{enumerate}
\end{proposition}

\begin{proof} Without loss of generality, up to a translation, we can  assume that $p=0$. Moreover, by Lemma \ref{Lem:speed-indip-orbit}, in order to get asymptotic estimates of $v(t)$ and $v^o(t)$, it is enough to estimate $\omega(z_0, \phi_t(z_0))$ for any suitably chosen $z_0\in \D$. Note that $\omega(z_0, \phi_t(z_0))=k_V(h(z_0), h(z_0)+it)$, where  $V:=V(\alpha,\beta)$. 

In case $\alpha, \beta>0$, we choose $h(z_0)=i$. Note that $V=R(W)$, where $R(z)=ie^{i(\beta-\alpha)/2}z$ and  
\[
W:=\{\rho e^{i\theta}: \rho>0, |\theta|<(\alpha+\beta)/2\}.
\]
Hence, taking into account that $h(z_0)=i$, we have 
\[
k_V(h(z_0), h(z_0)+it)=K_W(e^{i(\alpha-\beta)/2}, e^{i(\alpha-\beta)/2}(1+t)).
\]  

The map $f:W\to \Ha$ given by $f(w):=w^{\pi/(\alpha+\beta)}$ is a biholomorphism. Therefore, if we set $\theta_0:=\frac{\pi(\alpha-\beta)}{2(\alpha+\beta)}$, we have
\begin{equation*}
k_V(h(z_0), h(z_0)+it)=k_\Ha(e^{i\theta_0}, e^{i\theta_0}(1+t)^{\pi/(\alpha+\beta)}). 
\end{equation*}
Now,  by Lemma \ref{Lem:hyper-semipiano}(6),
\begin{equation*}
\begin{split}
&|k_\Ha(e^{i\theta_0}, e^{i\theta_0}(1+t)^{\pi/(\alpha+\beta)})- k_\Ha(1, (1+t)^{\pi/(\alpha+\beta)})|\\&\leq |k_\Ha(e^{i\theta_0}, e^{i\theta_0}(1+t)^{\pi/(\alpha+\beta)})- k_\Ha(1, e^{i\theta_0}(1+t)^{\pi/(\alpha+\beta)})|\\&+|k_\Ha(1, e^{i\theta_0}(1+t)^{\pi/(\alpha+\beta)})- k_\Ha(1, (1+t)^{\pi/(\alpha+\beta)})|\\&\leq k_\Ha(1, e^{i\theta_0})+k_\Ha(e^{i\theta_0}(1+t)^{\pi/(\alpha+\beta)}, (1+t)^{\pi/(\alpha+\beta)})|\\&\leq k_\Ha(1, e^{i\theta_0})+\frac{1}{2}\log \frac{2}{\cos \theta_0}.
\end{split}
\end{equation*}
Since $k_\Ha(1, (1+t)^{\pi/(\alpha+\beta)})=\frac{1}{2}\log (1+t)^{\pi/(\alpha+\beta)}$, the previous considerations show that there exists $C>0$ such that
\[
|k_V(h(z_0), h(z_0)+it)-\frac{\pi}{2(\alpha+\beta)}\log t|<C
\]
for all $t\geq 0$, and we are done in case $\alpha, \beta>0$. 

Now we assume that $\beta=0$ (the case $\alpha=0$ being similar). In this case, we choose $h(z_0)=e^{i (\pi-\alpha)/2}$ (note that $(0,+\infty)\ni t\mapsto te^{i (\pi-\alpha)/2}$ is the symmetry axis of $V$). 
Arguing as before, one can see that 
\begin{equation*}
k_V(h(z_0), h(z_0)+it)=k_W(1, 1+te^{i\alpha/2}). 
\end{equation*}
We write $1+te^{i\alpha/2}=\rho_t e^{i\theta_t}$. Since $f:W\to \Ha$ defined as $f(w)=w^{\pi/\alpha}$ is a biholomorphism, we have
\[
k_W(1, 1+te^{i\alpha/2})=k_\Ha(1, \rho_t^{\pi/\alpha} e^{i(\theta_t\pi)/\alpha}). 
\]
By Proposition \ref{Prop:distance-como-suma}, 
\[
|k_\Ha(1, \rho_t^{\pi/\alpha} e^{i(\theta_t\pi)/\alpha})-k_\Ha(1, \rho_t^{\pi/\alpha} )-k_\Ha(\rho_t^{\pi/\alpha} , \rho_t^{\pi/\alpha} e^{i(\theta_t\pi)/\alpha})|\leq \frac{1}{2}\log 2.
\]
Hence, we are left to compute $k_\Ha(1, \rho_t^{\pi/\alpha} )+k_\Ha(\rho_t^{\pi/\alpha} , \rho_t^{\pi/\alpha} e^{i(\theta_t\pi)/\alpha})$. By Lemma~\ref{Lem:hyper-semipiano}, we have
\[
k_\Ha(1, \rho_t^{\pi/\alpha} )=\frac{\pi}{2\alpha}\log \rho_t, \quad k_\Ha(\rho_t^{\pi/\alpha} , \rho_t^{\pi/\alpha} e^{i(\theta_t\pi)/\alpha})=k_\Ha(1, e^{i(\theta_t\pi)/\alpha}),
\]
and 
\[
|k_\Ha(1, e^{i(\theta_t\pi)/\alpha})-\frac{1}{2}\log \frac{1}{\cos(\frac{\theta_t\pi}{\alpha})}|<\frac{1}{2}\log 2.
\]
Therefore, there exists $C>0$ such that
\[
|k_V(h(z_0), h(z_0)+it)-\frac{\pi}{2\alpha}\log \rho_t-\frac{1}{2}\log \frac{1}{\cos(\frac{\theta_t\pi}{\alpha})}|<C.
\]
Now, 
\[
\rho_t=\sqrt{t^2+2\cos(\alpha/2)t+1}, \quad \cos \theta_t=\frac{1+\cos(\alpha/2)t}{\rho_t}.
\]
Clearly, $\lim_{t\to +\infty} \frac{\rho_t}{t}=1$, which implies that $\frac{\pi}{2\alpha}\log \rho_t$ goes like $\frac{\pi}{2\alpha}\log t$ as $t\to+\infty$. Let us analyze the asymptotic behavior of the term 
$\frac{1}{2}\log \frac{1}{\cos(\frac{\theta_t\pi}{\alpha})}$. 
Notice that $\lim_{t\to +\infty}\cos \theta_t=\cos(\alpha/2)$ and $\lim _{t\to +\infty} (\rho_t-t)=\cos(\alpha/2)$. Applying the Mean Value Theorem to the function $g(x)=\arccos(x)$, we deduce that for each $x\in [0,1]$ there is a point $\xi$ in the interval of extremes points $x$ and $\cos(\alpha/2)$ such that 
$$
g(x)-\frac{\alpha}{2}=g'(\xi) (x-\cos(\alpha/2)).
$$
Taking $x=\cos (\theta_t)$ we deduce that there is  $\xi_t$ in the interval of extremes points $\cos \theta_t$ and $\cos(\alpha/2)$ such that
$$
\theta_t-\frac{\alpha}{2}=-\frac{1}{\sqrt{1-\xi_t^2}}(\cos(\theta_t)-\cos(\alpha/2)).
$$ 
Clearly, we have that $\lim_{t\to +\infty} \xi_t=\cos(\alpha/2)$. Thus, 
$$
\lim_{t\to +\infty} \frac{\cos(\theta_t)-\cos(\alpha/2)}{\theta_t-\frac{\alpha}{2}}=-\lim_{t\to +\infty} \sqrt{1-\xi_t^2}=-\sin (\alpha/2).
$$
Therefore
\begin{equation*}
\begin{split}
\lim_{t\to +\infty} t\cos(\frac{\theta_t\pi}{\alpha})&=\frac{\pi}{\alpha}\lim_{t\to +\infty} t\, \left(\theta_t -\frac{\alpha}{2}\right)\, 
\frac{\cos(\frac{\theta_t\pi}{\alpha})}{\theta_t \frac{\pi}{\alpha}-\frac{\pi}{2}}=-\frac{\pi}{\alpha}\lim_{t\to +\infty} t\, \left(\theta_t -\frac{\alpha}{2}\right)\\
&=\frac{\pi}{\alpha\, \sin (\alpha/2)}\lim_{t\to +\infty} t\, \left(\cos \theta_t -\cos (\alpha/2)\right)\\
&=\frac{\pi}{\alpha\, \sin (\alpha/2)}\lim_{t\to +\infty} \frac{t}{\rho_t}\, \left(1+\cos (\alpha/2)(t-\rho_t)\right)\\
&=\frac{\pi}{\alpha\, \sin (\alpha/2)}\left(1-\cos^2 (\alpha/2)\right)=\frac{\pi}{\alpha}\sin (\alpha/2)\in (0,+\infty).
\end{split}
\end{equation*}
Thus, $\frac{1}{2}\log \frac{1}{\cos(\frac{\theta_t\pi}{\alpha})}$ goes like $\frac{1}{2}\log t$ as $t\to+\infty$ and the result follows.
\end{proof}

In Proposition \ref{Prop:total-speed} we showed that if $(\phi_t)$ is a parabolic semigroup in $\D$, then $v(t)/t\to 0$ as $t\to+\infty$. This is essentially the only possible upper bound, as the following proposition shows:

\begin{proposition}
Let $g:[0,+\infty)\to [0,+\infty)$ be a function such that $\lim_{t\to+\infty}g(t)=+\infty$ and $\lim_{t\to+\infty}\frac{g(t)}{t}=0$. Then there exists a parabolic semigroup  $(\phi_t)$ in $\D$ of zero hyperbolic step such that
\[
\limsup_{t\to+\infty}\frac{v(t)}{g(t)}=+\infty.
\]
\end{proposition}
\begin{proof}
Let $\{a_j\}$ be a strictly increasing sequence of  positive real numbers, $a_1>0$, $\lim_{j\to+\infty} a_j=+\infty$. Let $\{b_j\}$ be a strictly increasing sequence of positive real numbers to be chosen later on. Let 
\[
\Omega:=\C\setminus \left(\bigcup_{j=1}^{\infty}\{z\in \C: \Re z=\pm a_j, \Im z\leq b_j\}\right).
\]
Note that $\Omega$ is simply connected and starlike at infinity. Let $h:\D \to \Omega$ be a Riemann map such that $h(0)=0$, and let $\phi_t(z):=h^{-1}(h(z)+it)$ for $z\in \D$ and $t\geq 0$. Then $(\phi_t)$ is a semigroup in $\D$ and, since $\bigcup_{t\geq 0}(\Omega-it)=\C$, it follows that $(\phi_t)$ is parabolic of zero hyperbolic step. 

In order to estimate the total speed $v(t)$ of $(\phi_t)$, note that $\Omega$ is symmetric with respect to the imaginary axis $i\R$, hence  the orbit $[0,+\infty)\ni t\mapsto it$ is a geodesic in $\Omega$, and so is $[0,+\infty)\ni t\mapsto \phi_t(0)$ in $\D$. 

In particular, if we set $\gamma(t)=it$, we have
\begin{equation}\label{Eq:total-speed-semi-fast1}
\begin{split}
v(t)&=\omega(0, \phi_t(0))=k_\Omega(0, it)=\int_0^t \kappa_\Omega(\gamma(r);\gamma'(r))dr\\&\geq \frac{1}{4}\int_0^t \frac{dr}{\delta_\Omega(ir)},
\end{split}
\end{equation}
where the last inequality follows from the classical estimates on the hyperbolic metric (see, {\sl e.g.}, \cite{BCGD})

Now, we claim that we can choose the $b_j$'s in such a way that for every $j\geq 1$ there exists $x_j\in (b_j, b_{j+1})$ such that $\delta_\Omega(it)=a_{j+1}$ for every $t\in [x_j, b_{j+1}]$ and such that
\begin{equation}\label{Eq:good-bj-cres}
b_{j+1}-x_j\geq j a_{j+1}g(b_{j+1}).
\end{equation}
Indeed, set $b_1=1$. Let $x_1>1$ be such that $|ix_1-(a_1+ib_1)|=a_2$. Simple geometric consideration shows that, if we take $b_2>x_1$ then $\delta_\Omega(it)=a_2$ for every $t\in [x_1, b_2]$. Moreover, since $g(t)/t\to 0$ as $t\to+\infty$, we can find $b_2>x_1$ such that 
\[
\frac{a_2 g(b_2)+x_1}{b_2}<1. 
\]
Therefore, there exist $x_1, b_2$ such that \eqref{Eq:good-bj-cres} is satisfied for $j=1$. Now, we can argue by induction is a similar way. Suppose we constructed $b_1, \ldots, b_j$ and $x_1, \ldots, x_{j-1}$ for $j>1$. Then we select $x_j$ in such a way that $|ix_j-(a_j+ib_j)|=a_{j+1}$ and, again since $g(t)/t\to 0$ as $t\to+\infty$, we choose $b_{j+1}>x_j$ such that $\frac{j a_{j+1} g(b_{j+1})+x_j}{b_{j+1}}<1$.

Thus, by \eqref{Eq:total-speed-semi-fast1} and \eqref{Eq:good-bj-cres}, we have
\[
v(b_{j+1})\geq \frac{1}{4}\int_0^{b_{j+1}}\frac{dr}{\delta_\Omega(ir)}\geq \frac{1}{4}\int_{x_j}^{b_{j+1}}\frac{dr}{a_{j+1}}=\frac{b_{j+1}-x_j}{4a_{j+1}}\geq \frac{j g(b_{j+1})}{4}.
\]
Therefore,
\[
\frac{v(b_{j+1})}{g(b_{j+1})}\geq \frac{j}{4},
\]
hence $\limsup_{t\to+\infty}\frac{v(t)}{g(t)}=+\infty$, and we are done.
\end{proof}

\section{Orthogonal speed of convergence of parabolic semigroups}
\label{Sec:velocidad ortogonal}

In this section we give estimates on the orthogonal speed of convergence of semigroups. Since the orbits of hyperbolic semigroups converge non-tangentially to the Denjoy-Wolff point, it follows from \eqref{Eq:split-speed-semig} that the total and the orthogonal speeds of hyperbolic semigroups have the same asymptotic behavior. Therefore, we concentrate on parabolic semigroups.

In order to simplify the notation, for any $\alpha\in (0,\pi]$, we write 
$$
V(\alpha):=V(\alpha,0)=\{w=\rho e^{i\theta}:\, \rho>0, \, |\theta|<\alpha\}.
$$

The first part of the following result follows immediately from the fact that $h(\D)$ is contained in the Koebe domain $\C\setminus\{z\in\C: \Re z=\Re p, \Im z\leq \Im p\}$, where $p\in \C\setminus h(\D)$ and Proposition \ref{Prop:precise-est-angles}. While, the second part is a deep result in \cite{BetCD}, where the analogue Euclidean expression is estimated using harmonic measure theory (and then the result in terms of speed follows from Proposition \ref{Prop:expre-speed-in-disc}).

\begin{theorem}\label{Thm:sector}
Let $(\phi_t)$ be a parabolic semigroup, not a group, in $\D$ with Denjoy-Wolff point $\tau\in \partial \D$ and Koenigs function $h$.
 Suppose that $h(\D)$ is contained in a
sector $p+iV(\alpha)$, $p\in\C,\;\alpha\in (0,\pi]$. Then
\begin{equation*}
\liminf_{t\to+\infty}[v(t)-\frac{\pi}{4\alpha}\log t]>-\infty,
\end{equation*}
and,
\begin{equation*}
\liminf_{t\to+\infty}[v^o(t)-\frac{\pi}{4\alpha}\log t]>-\infty.
\end{equation*}
\end{theorem}

\begin{remark}
The previous bounds are sharp, as shown by Proposition \ref{Prop:precise-est-angles}.
\end{remark}

In general, we have the following  bounds (which was proved in its Euclidean counterpart by D. Betsakos \cite{Bet}):

\begin{theorem}\label{Thm:betsakos}
Let $(\phi_t)$ be a parabolic semigroup in $\D$ with Denjoy-Wolff point $\tau\in \partial \D$. 
\begin{enumerate}
\item $\liminf_{t\to+\infty}[v^o(t)-\frac{1}{4}\log t]>-\infty$.
\item If, in addition, the semigroup is of positive hyperbolic step, then $\liminf_{t\to+\infty}[v^o(t)-\frac{1}{2}\log t]>-\infty$.
\end{enumerate}
\end{theorem}
\begin{proof} Let $(\Omega, h, z\mapsto z+it)$ be the canonical model of the semigroup.

(1) Take a point $p\in \C\setminus h(\D)$. Since $h$ is starlike at infinity, $h(\D)\subset p+iV(\pi)$ and the result follows immediately from  Theorem \ref{Thm:sector}. 

(2) By \eqref{Eq:split-speed-semig} and \eqref{Eq:split-speed-semig2} we have 
\[
v(t)\leq 2v^o(t)+4\log 2.
\]
Hence, by Proposition \ref{Prop:total-speed-lowerb},
\[
\liminf_{t\to +\infty}[v^o(t)-\frac{1}{2}\log t]\geq \frac{1}{2}\liminf_{t\to +\infty}[v(t)-\log t-2\log2]>-\infty.
\]
\end{proof}

 \begin{remark}
 The  bounds given by Theorem \ref{Thm:betsakos} are sharp (see Proposition \ref{Prop:precise-est-angles}).
\end{remark} 

\begin{remark}
Proposition \ref{Prop:total-speed-lowerb}, Theorem \ref{Thm:betsakos} and  \eqref{Eq:split-speed-semig} imply at once that if $(\phi_t)$ is a non-elliptic semigroup in $\D$ and there exists a constant $C>0$ such that for all $t\geq 0$
\[
|v^o(t)-\frac{1}{4}\log t|<C,
\]
then $\limsup_{t\to+\infty}v^T(t)<+\infty$ and hence $[0,+\infty)\ni t\mapsto \phi_t(z)$ converges non-tangentially to the Denjoy-Wolff point for every $z\in \D$.
\end{remark} 

\section{Open Questions}
 
 The previous results give rise to the following questions:
 
 \medskip
 
 {\sl Question 1:} Suppose $(\phi_t)$ is a non-elliptic semigroup in $\D$. Is it true that $\limsup_{t\to\infty}[v^T(t)-\frac{1}{2}\log t]<+\infty$?
 
 \medskip
 
  {\sl Question 2:} Suppose $(\phi_t)$ is a parabolic semigroup in $\D$ of positive hyperbolic step. Is it true that $|v^T(t)-\frac{1}{2}\log t|<C$ for some constant $C>0$? If so, does this condition characterize parabolic semigroups of positive hyperbolic step?
  
 \medskip
 
  {\sl Question 3:} Suppose $(\phi_t)$ is a parabolic semigroup in $\D$. Is it possible to characterize in dynamical terms when $\lim_{t\to\infty}v^T(t)=\infty$ and $\lim_{t\to\infty}\frac{v^T(t)}{v^o(t)}=0$?
  
  \medskip
 
  {\sl Question 4:} Is it true that the orthogonal speed is essentially monotone? Namely, suppose $(\phi_t)$, $(\tilde\phi_t)$ are a parabolic semigroup in $\D$ with Koenigs' functions $h$ and $\tilde h$ and orthogonal speeds $v^o(t)$ and $\tilde v^0(t)$ respectively. Suppose $h(\D)\subset \tilde h(D)$. Is it true that $\liminf_{t\to\infty}[v^o(t)-\tilde v^o(t)]>-\infty$?
  
 \medskip
 
  {\sl Question 5:} Does there exist a non-elliptic semigroup whose total speed ({\sl respect.} orthogonal speed) does not have a precise asymptotic value? Namely, does there exist a parabolic semigroup such that $\limsup_{t\to\infty}\frac{v(t)}{g(t)}=\infty$ and $\liminf_{t\to\infty}\frac{v(t)}{g(t)}=0$ for some function $g:[0,+\infty)\to [0,+\infty)$ such that $\lim_{t\to \infty}g(t)=\infty$?


\begin{thebibliography}{99}
\bibitem{Ababook89} M. Abate, {\sl Iteration theory of holomorphic maps on taut manifolds}, Mediterranean Press, Rende, 1989.

\bibitem{A1} D. Aharonov, M. Elin, S. Reich and D. Shoikhet, {\sl Parametric representation of semi-complete vector fields on the unit balls of $\C^n$ and in Hilbert space}, Atti Accad. Naz. Lincei 10 (1999), 229--253.




\bibitem{AroBra16} L. Arosio, F. Bracci, {\sl Canonical models for holomorphic iteration}. Trans. Amer. Math. Soc., 5, \textbf{368}, (2016), 3305--3339.


\bibitem{BerPor78} E. Berkson, H. Porta, \textsl{Semigroups of
holomorphic functions and composition operators}. Michigan
Math. J., \textbf{25} (1978), 101--115.

\bibitem{Bet} D. Betsakos, {\sl On the rate of convergence of parabolic semigroups of holomorphic functions}.
Anal. Math. Phys., 5 (2015), 207-216.

\bibitem{BetCD} D. Betsakos, M. D. Contreras, S. D\'iaz-Madrigal, {\sl On the rate of convergence of semigroups
of holomorphic functions at the Denjoy-Wolff point}. To appear in Rev. Mat. Iberoamericana.


\bibitem{A2} F. Bracci, P. Gumenyuk, {\sl Contact points and fractional singularities for semigroups of holomorphic self-maps in the unit disc}.  J. Anal. Math., 130, (2016), 1, 185-217.

\bibitem{BCD} F. Bracci, M. D. Contreras, S. D\'iaz-Madrigal, {\sl Topological invariants for semigroups of holomorphic self-maps of the unit disc}.  J. Math. Pures Appl., 107, 1, (2017), 78-99.

\bibitem{BCD2} F. Bracci, M. D. Contreras, S. Diaz-Madrigal, {\sl On the Koenigs function of semigroups of holomorphic self-maps of the unit disc}. Anal. Math. Phys. 8 (2018), no. 4, 521--540.

\bibitem{BCGD} F. Bracci, M. D. Contreras, S. Diaz-Madrigal, H. Gaussier, {\sl Backward orbits and petals of semigroups of holomorphic self-maps of the unit disc}. Ann. Math. Pura Appl. 198, 2, (2019), 411-441.

\bibitem{BCDG2} F. Bracci, M. D. Contreras, S. Diaz-Madrigal, H. Gaussier, {\sl Non-tangential limits and the slope of trajectories of holomorphic semigroups of the unit disc}. Preprint 2018 

\bibitem{BCGDZ} F. Bracci, M. D. Contreras, S. Diaz-Madrigal, H. Gaussier, A. Zimmer, {\sl Asymptotic behavior of orbits of holomorphic semigroups}, J. Math. Pures Appl. doi:10.1016/j.matpur.2019.05.005  online print



\bibitem{ColLohbook66} E.F. Collingwood, A.J. Lohwater, {\sl The theory of cluster sets}, Cambridge Tracts in Mathematics and Mathematical Physics, No. 56 Cambridge Univ. Press, Cambridge, 1966.


\bibitem{ConDia05a} M. D. Contreras, S. D\'iaz-Madrigal, {\sl Analytic flows on the unit disk: angular derivatives and boundary fixed points}. Pacific J. Math., \textbf{222}  (2005), 253--286.

\bibitem{CoDiPo04} M. D. Contreras, S. D\'iaz-Madrigal, and Ch. Pommerenke, {\sl Fixed points and boundary behavior of the Koenigs function}. Ann. Acad. Sci. Fenn. Math., \textbf{29} (2004), 471--488.

\bibitem{A3} M. D. Contreras, S. D\'iaz-Madrigal, Ch. Pommerenke, {\sl On boundary critical points for semigroups of analytic functions}. Math. Scand. 98 (2006), 125--142.


\bibitem{Cow81} C. C. Cowen, {\sl Iteration and the solution of functional equations for functions analytic in the unit disk}, Trans. Amer. Math. Soc. \textbf{265} (1981), 69--95.



\bibitem{A6} F. Jacobzon, S. Reich and D. Shoikhet, {\sl Linear fractional mappings, invariant sets, semigroups and commutativity}, J. Fixed Point Theory Appl. 5 (2009), 63--91. 

\bibitem{A7} F. Jacobzon, M. Levenshtein, S. Reich, {\sl Convergence characteristics of one-parameter continuous semigroups}, Analysis and Mathematical Physics, 1 (2011), 311--335.   


\bibitem{ElKhReSh10b2} M. Elin, D. Khavinson, S. Reich and D. Shoikhet, {\sl Linearization models for parabolic dynamical systems via Abel's functional equation}. Ann. Acad. Sci. Fen. Math., 35 (2010), 439--472.

\bibitem{A4} M. Elin, M. Levenshtein, S. Reich and D. Shoikhet, {\sl Commuting semigroups of holomorphic mappings}, Math. Scad. 103 (2008), 295--319. 


\bibitem{EliShobook10} M.\,Elin and D.\,Shoikhet, Linearization models for complex dynamical systems. Topics in univalent functions, functional equations and semigroup theory. Birkh\"auser Basel, 2010.

\bibitem{A5} M. Elin, S. Reich, D. Shoikhet and F. Yacobzon, {\sl Rates of convergence of one-parameter semigroups with boundary Denjoy-Wolff fixed points}, in Fixed Points Theory and its Applications, Yokohama Publishers, Yokohama, 2008, 43--58. 



\bibitem{ElShZa08} M. Elin, D. Shoikhet, L. Zalcman, {\sl A flower structure of backward flow invariant domains for semigroups}. Ann. Acad. Sci. Fenn. Math., \textbf{33} (2008), 3--34.











\bibitem{Shobook01} D. Shoikhet, {\sl Semigroups in geometrical function theory}. Kluwer Academic Publishers, Dordrecht, 2001.

\bibitem{Sis85} A. G. Siskakis, {\sl Semigroups of composition operators and the Ces\`{a}ro operator on $H^{p}(D)$}. Ph.\,D.
Thesis, University of Illinois, 1985

\bibitem{Sis98} A. G. Siskakis, {\sl Semigroups of composition operators
on spaces of analytic functions, a review,} 229--252. \emph{Contemp. Math.}
vol. 213, Amer. Math. Soc., Providence, RI, 1998.


\end{thebibliography}
\end{document}